\documentclass[11pt]{article}

\usepackage{amssymb}
\usepackage{latexsym}
\usepackage{amsmath}
\usepackage{verbatim}
\usepackage{amsthm}
\usepackage{algorithm}
\usepackage[noend]{algorithmic}
\usepackage{graphicx}
\usepackage{xcolor}
\usepackage{epsfig}
\usepackage{cite}

\newcommand{\Z}{\mathbb{Z}}
\newcommand{\card}{\mbox{card}}
\newcommand{\Per}{\mbox{Per}}
\renewcommand{\d}{\ensuremath{\diamond}}

\numberwithin{equation}{section}

\title{Pattern occurrence statistics and applications to the\\ Ramsey theory of
unavoidable patterns\thanks{This material is based upon work supported by the National Science Foundation under Grant No. DMS--1060775.}}

\author{Jim Tao$^1$}

\newtheorem{thm}{Theorem}
\newtheorem{lma}{Lemma}
\newtheorem{cor}{Corollary}

\newtheorem{df}{Definition}
\newtheorem{ex}{Example}
\newtheorem{prob}{Problem}

\newtheorem{conj}{Conjecture}

\begin{document}

\maketitle

\footnotetext[1]{Department of Mathematics, California Institute of Technology, 1200 E California Blvd. Mathematics 253-37, Pasadena, CA 91125, USA,
{\tt jtao@caltech.edu}}

\begin{abstract}
As suggested by Currie, we apply the probabilistic method to problems
regarding pattern avoidance. Using techniques from analytic combinatorics,
we calculate asymptotic pattern occurrence statistics and use them in
conjunction with the probabilistic method to establish new results about
the Ramsey theory of unavoidable patterns in the full word case (both nonabelian sense and abelian sense) and in the partial word case. 

\end{abstract}

\section{Introduction}

In \cite{currie}, Currie reviews results and formulates a large number of open problems concerning pattern avoidance as well as an abelian variation of it. Given a pattern $p$ over an alphabet $\mathcal{V}$ and a word $w$ over an alphabet $\mathcal{A}$, we say that $w$ {\em encounters} $p$ if there exists a nonerasing morphism $h: \mathcal{V}^* \rightarrow \mathcal{A}^*$ such that $h(p)$ is a factor of $w$; otherwise $w$ {\em avoids} $p$. In other words, $w$ encounters $p=p_1 \cdots p_n$, where $p_1, \ldots, p_n \in \mathcal{V}$,  if $w$ contains $u_1 \cdots u_n$ as a factor, where $u_1, \ldots, u_n$ are nonempty words in $\mathcal{A}^*$ satisfying $u_i=u_j$ whenever $p_i=p_j$. On the other hand,  $w$ {\em encounters} $p=p_1 \cdots p_n$ {\em in the abelian sense}  if $w$ contains $u_1 \cdots u_n$ as a factor, where $u_i$ can be obtained from $u_j$ by rearranging letters whenever $p_i=p_j$; otherwise $w$ {\em avoids} $p$ {\em in the abelian sense}.

Words avoiding patterns such as squares have been used to build several counterexamples in context-free languages \cite{MaBuHa}, groups \cite{AdNo}, lattice of varieties \cite{Jez}, partially ordered sets \cite{TrWi}, semigroups \cite{Jus,LuVa}, symbolic dynamics \cite{Mor},  to name a few. Words avoiding squares in the abelian sense have also been used in the study of free partially commutative monoids \cite{Cor,Dic}, and have helped characterize the repetitive commutative semigroups \cite{Jus}. In addition, words avoiding more general patterns find applications in algorithmic problems on algebraic structures \cite{KhSa}.  

In this paper, we meet the goal of Problem~4 as expressed by Currie in \cite{currie}, which is to ``explore the scope of application of the probabilistic method to problems in pattern avoidance.'' The probabilistic method \cite{probabilistic}, pioneered by Erd\H{o}s, has recently become one of the most powerful techniques in combinatorics. It is used to demonstrate, via statistical means, the existence of certain combinatorial objects without constructing them explicitly. Analytic combinatorics \cite{flajolet}, pioneered by Flajolet and Sedgewick and expanded to the multivariate case \cite{pemantle} by Pemantle and Wilson, allows precise calculation of the statistics of large combinatorial structures by studying their associated generating functions through the lens of complex analysis. Since analytic combinatorics calculates the statistics of large combinatorial structures, and the probabilistic method uses such statistics to infer the existence of specific combinatorial objects, we use both techniques in tandem to prove some Ramsey theoretic results about pattern
avoidance. 

We also extend some of our results to partial words, which allow for undefined positions represented by hole characters. In this context, given a pattern $p$ over $\mathcal{V}$ and a partial word $w$ over $\mathcal{A}$, we say that $w$ {\em encounters} $p$ if there exists a nonerasing morphism $h: \mathcal{V}^* \rightarrow \mathcal{A}^*$ such that $h(p)$ is {\em compatible with} a factor of $w$. Several results concerning (abelian) pattern avoidance have recently been proved in this more general context of partial words (see, for example, \cite{BSDeWSi,BSKiMeSeSiXu,BSLoSc,BSMeSiWe,BSSiXu,BSWo}).

The contents of our paper are as follows.
In Section~2, we discuss some basic concepts and fix some notations. 
In Section~3, we discuss some tools, such as the ordinary generating functions, and techniques from analytic combinatorics. 
In Sections~4 and 5, we use those tools and techniques in conjunction with the probabilistic method to calculate asymptotic pattern occurrence statistics and to establish new results about the Ramsey theory of unavoidable patterns in the full word case (both nonabelian sense and abelian sense) and the partial word case. 
Finally in Section~6, we suggest additional possible uses of these data in applications such as cryptography and musicology. We also discuss a number of open problems.

\section{Basic concepts and notations}

A (full) {\em word} over an alphabet $\mathcal A$ is a sequence of characters
from $\mathcal A$. We call the characters in $\mathcal A$ {\em letters}.
The number of characters in a word is its {\em length}. We denote by $\mathcal{A}^*$
the set of all words over $\mathcal A$; when equipped with the concatenation or product of words,
where the empty word $\varepsilon$ serves as identity, it is called
the free monoid generated by $\mathcal A$. A word $w$ over $\mathcal A$ {\em encounters}
the word $p$ over an alphabet $\mathcal V$ if $w$ contains $h(p)$ as
a factor for some nonerasing morphism $h:\mathcal V^*\rightarrow\mathcal A^*$.
Otherwise $w$ {\em avoids} $p$ and is {\em p-free}. In this case
we interpret $p$ to be a {\em pattern}. For example, the word $tennessee$ encounters the pattern $abaca$, as witnessed by
the morphism $h:\{a,b,c\}^*\rightarrow\{e,n,s,t\}^*$ with $h(a)=e$,
$h(b)=nn$, and $h(c)=ss$. Thus $tennessee$ contains $h(abaca)=ennesse$, a factor of tennessee.

We count multiple instances of a pattern in a word as follows: we say that $w$ encounters
$p$ a total of $N>0$ times if, for some maximal $m>0$, there exist $m$ distinct nonerasing morphisms
$h_i:\mathcal V^*\rightarrow\mathcal A^*$ such that for some $t_1,\ldots,t_m>0$, $h_i(p)$ is a factor of $w$ exactly
$t_i>0$ times, and $\sum_{i=1}^m t_i=N$. For example, the word $11111111$ encounters the pattern
$aba$ $34$ times because for $3\le k\le 8$, each of the $9-k$ factors of length $k$ lies in
the image of $\lfloor (k-1)/2\rfloor$ nonerasing morphisms
$\{a,b\}^*\rightarrow\{1\}^*$, and $6\cdot 1+5\cdot 1+4\cdot 2+3\cdot 2+2\cdot 3+1\cdot 3=34$.
One may object to this definition on the basis that the factor $11111111$ is counted as three occurrences
of the pattern $aba$, but since pattern occurrences are defined in terms of nonerasing morphisms, it makes sense
to count the same factor multiple times if it lies in the image of multiple distinct nonerasing morphisms.
Patterns are an abstract idea that goes beyond the concrete words that they map to under these nonerasing
morphisms; they are a kind of symmetry that exists in the words in which they appear. For that reason
the $aba$ subgroup of the ``symmetry group'' of the factor $11111111$ should be larger than that of the factor
$12345671$, just as the group of symmetries of a circle is larger than that of a square. As Hermann Weyl once said,
\begin{quote}
What has indeed become a guiding principle in modern mathematics is this lesson: \textit{Whenever you have to do with
a structure-endowed entity $\Sigma$ try to determine its group of automorphisms}, the group of those
element-wise transformations which leave all structural relations undisturbed. You can expect to gain a deep
insight into the constitution of $\Sigma$ in this way. After that you may start to investigate symmetric
configurations of elements, i.e. configurations which are invariant under a certain subgroup of the group
of all automorphisms; and it may be advisable, before looking for such configurations, to study the subgroups
themselves. \cite[p.~144]{weyl}
\end{quote}
Our definition allows a kind of intuition analogous
to counting rectangles in a rectangular grid. One may object that some rectangles are equivalent up
to similarity, but there is no reason to make problems harder than they need to be.

A {\em partial word} over $\mathcal A$ is a sequence of
characters from the extended alphabet $\mathcal A+\{\diamond\}$, where 
we refer to $\diamond$ as the {\em hole} character. Define the {\em hole density}
of a partial word to be the ratio of its number of holes to its length,
i.e. $d:=h/n$ where $d$ is the hole density, $h$ is the number of holes,
and $n$ is the length of the partial word. A {\em completion} of a partial word $w$ is a full word constructed by filling in the holes of $w$ with letters from $\mathcal A$.

If $u=u_1\cdots u_n$
and $v=v_1\cdots v_n$ are partial words of equal length $n$, where
$u_1,\ldots,u_n$ and $v_1,\ldots,v_n$ denote characters from
$\mathcal A+\{\diamond\}$, we say that $u$ is {\em compatible} with
$v$, denoted $u\uparrow v$, if $u_i=v_i$ whenever $u_i,v_i\in\mathcal A$.
A partial word $w$ over $\mathcal A$
encounters the full word $p$ over $\mathcal V$ if some factor $f$
of $w$ satisfies $f\uparrow h(p)$ for some nonerasing morphism
$h:\mathcal V^*\rightarrow\mathcal A^*$. 
Otherwise $w$ {\em avoids} $p$ and is {\em p-free}. Again
we interpret $p$ to be a {\em pattern}. For example, the partial word $velve{\diamond}ta$ encounters $abab$,
as witnessed by the morphism $h:\{a,b\}^*\rightarrow\{a,e,l,v,t\}^*$ with
$h(a)=ve$ and $h(b)=l$. Thus $h(abab)=velvel$, which is compatible
with $velve{\diamond}$, a factor of $velve{\diamond}ta$.
We count multiple instances of a pattern in a partial word as follows: we say that $w$ encounters
$p$ a total of $N>0$ times if, for some maximal $m>0$, there exist $m$ distinct nonerasing morphisms
$h_i:\mathcal V^* \rightarrow\mathcal A^*$ such that for some $t_1,\ldots,t_m>0$, there are $t_i>0$ factors $f_i$
of $w$ that satisfy $f_i\uparrow h_i(p)$, and $\sum_{i=1}^m t_i=N$.

Suppose $p=p_1\cdots p_n$ where $p_1,\ldots,p_n \in \mathcal{V}$.
A full word $w$ {\em encounters} $p$ {\em in the abelian sense} if $w$
contains $u_1\cdots u_n$ as a factor, where word $u_j$ can be obtained
from word $u_k$ by rearranging letters whenever $p_j=p_k$. Otherwise
$w$ {\em avoids} $p$ {\em in the abelian sense} and is {\em abelian p-free}.
For example, the full word $v$ $al$ $h$ $al$ $la$ encounters $abaa$
in the abelian sense. We count multiple instances of an abelian pattern
in a word as follows: we say that $w$ encounters $p$ in the abelian sense
$N>0$ times if, for some maximal $m>0$, there exist $m$ distinct sequences
of words $S_i$ of the form $(u_1,\ldots,u_n)$ such that $w$ contains
$u_1\cdots u_n$ as a factor $t_i>0$ times, word $u_j$ can be obtained
from word $u_k$ by rearranging letters whenever $p_j=p_k$, and $\sum_{i=1}^m t_i=N$.

A pattern $p$ is {\em $m$-avoidable} if there are arbitrarily long words over
an $m$-letter alphabet that avoid $p$. A pattern $p$ is {\em $m$-avoidable
over partial words} if for every $h\in\mathbb N$ there is
a partial word with $h$ holes over an $m$-letter alphabet that avoids $p$.
A pattern $p$ is {\em $m$-avoidable in the abelian sense} if there are
arbitrarily long words over an $m$-letter alphabet that avoid $p$ in
the abelian sense. Otherwise, $p$ is, respectively, $m$-{\em unavoidable}, $m$-{\em unavoidable over partial words}, and $m$-{\em unavoidable in the abelian sense}. For example, 
the {\em Zimin patterns} $Z_i$ where 
\begin{equation}
Z_1=a_1 \text { and } Z_i=Z_{i-1}a_iZ_{i-1}\label{eq21}
\end{equation}
are $m$-unavoidable for all $m\ge 1$ \cite{lothaire}. They are also $m$-unavoidable over partial words for all $m\ge 1$ as well as $m$-unavoidable in the abelian sense for all $m\ge 1$. Indeed, since $Z_i$ occurs in a partial word whenever it occurs in some completion of the partial word, $Z_i$ is unavoidable over partial words, and since all occurrences of $Z_i$ in the nonabelian sense are occurrences of $Z_i$
in the abelian sense, $Z_i$ is unavoidable in the abelian sense.

Define the {\em Ramsey length} $L(m,p)$ of an $m$-unavoidable pattern $p$ to be
the minimal length of a word over an $m$-letter alphabet that ensures the
occurrence of $p$. Similarly, define the {\em partial Ramsey length} $L_d(m,p)$
of a pattern $p$ that is $m$-unavoidable over partial words with hole density
$\ge d$ to be the minimal length of a partial word with hole density $d$ over
an $m$-letter alphabet that ensures the occurrence of $p$, and define the
{\em abelian Ramsey length} $L_{\mathrm{ab}}(m,p)$ of a pattern $p$ that is
$m$-unavoidable in the abelian sense to be the minimal length of a word over
an $m$-letter alphabet that ensures the occurrence of $p$ in the abelian sense.

We use Knuth's {\em up-arrow notation} \cite{uparrow} defined as follows:
For all integers $x,y,n$ such that $y\ge 0$ and $n\ge 1$:
\begin{equation*}
x\uparrow^n y
=\left\{
\begin{matrix}
x^y &\textrm{ if }n=1, \\
1 &\textrm{ if }y=0, \\
x\uparrow^{n-1}(x\uparrow^n(y-1)) &\textrm{ otherwise.}
\end{matrix}
\right.
\end{equation*}
More specifically, we use the double up-arrow, i.e. the above operator where
$n=2$. For example, $3\uparrow\uparrow 3=3^{3^3}$.
We make use of the following identity regarding double up-arrows 
\begin{equation}
x\uparrow\uparrow(n+1)=x^{x\uparrow\uparrow n}, \label{eq22}
\end{equation}
which follows by induction on $n\ge 0$. Using the same example as before, we observe that $3\uparrow\uparrow 3=3^{3\uparrow\uparrow 2}=3^{3^3}$.

\section{Tools and techniques}

%We discuss the symbolic method, complex asymptotics, and the probabilistic method.

%\subsection{Symbolic method}

For standard terms and theorems related to the symbolic method, we refer the reader to the book of Flajolet and Sedgewick \cite{flajolet}.

%We make use of ordinary generating functions, defined as follows.

\begin{comment}
\begin{df}{\cite[p.~19]{flajolet}}
The {\em ordinary generating function} (OGF) 
of a sequence $\{A_n\}_{n=0}^{\infty}$ is the formal power series
$$A(z)=\sum_{n=0}^{\infty}A_nz^n.$$ The {\em ordinary generating function} (OGF)
of a combinatorial class $\mathcal A$ is the generating function of the sequence of numbers
$A_n=\card(\mathcal A_n)$, where $\mathcal A_n$ denotes the set of objects
in class $\mathcal A$ that have cardinality or size $n$. Equivalently, the OGF of class
$\mathcal A$ {\em admits} the combinatorial form
$$A(z)=\sum_{\alpha\in\mathcal A}z^{|\alpha|},$$
where $|\alpha|$ denotes the size of $\alpha$. We say that the variable $z$ marks size in the generating function.
\end{df}
\end{comment}

We can often specify a combinatorial class by performing a series of operations on basic ``atomic'' objects of size $1$: cartesian product $\mathcal B\times\mathcal C$, combinatorial sum (disjoint union) $\mathcal B+\mathcal C$, sequence construction $\mathrm{SEQ}(\mathcal B)$, and substitution $\mathcal B\circ\mathcal C$, where $\mathcal B, \mathcal C$ are combinatorial classes \cite[pp.~25--26,~87]{flajolet}.
As it turns out, specifications of combinatorial classes translate directly
into generating functions. According to the admissibility theorem for
ordinary generating functions \cite[pp.~27,~87]{flajolet}, the OGFs of such classes admit convenient
closed-form expressions. %We state it below for the cases relevant to us.

%\begin{thm}\cite[pp.~27,~87]{flajolet}
%The constructions of cartesian product, combinatorial sum, sequence, and substitution are all admissible. The associated operators are as follows:
%\begin{center}
%\begin{tabular}{llcl}
%Cartesian product:\hspace{.25cm} & $\mathcal A=\mathcal B\times\mathcal C$ & implies & $A(z)=B(z)\cdot C(z)$ \\
%Combinatorial sum: & $\mathcal A=\mathcal B+\mathcal C$ & implies & $A(z)=B(z)+C(z)$ \\
%Sequence construction: & $\mathcal A=\mathrm{SEQ}(\mathcal B)$ & implies &
%$A(z)=1/[1-B(z)]$ \\
%Substitution: & $\mathcal A=\mathcal B\circ\mathcal C$ & implies &
%$A(z)=B(C(z))$.
%\end{tabular}
%\end{center}
%For the sequence translation, it is assumed that $\mathcal B_0=\emptyset$,
%i.e. $\mathcal B$ does not have any objects of size zero.
%The class $\mathcal E=\{\varepsilon\}$ consisting of the neutral object only,
%and the class $\mathcal Z$ consisting of a single ``atomic'' object
%(node, letter) of size $1$ have OGFs $E(z)=1$ and $Z(z)=z$, respectively.
%\end{thm}

We recall some basic constructions \cite[p.~50]{flajolet}:
The class $\mathcal E=\{\varepsilon\}$ consisting of the neutral object only,
and the class $\mathcal Z$ consisting of a single ``atomic'' object
(node, letter) of size $1$ have OGFs $E(z)=1$ and $Z(z)=z$, respectively.
Let $\mathcal A=m\mathcal Z$ denote an
alphabet of $m$ letters and $\mathcal W=\text{SEQ}(\mathcal A)$
denote the set of all possible words over $\mathcal A$.
Then $\mathcal A$ and $\mathcal W$ have associated
OGFs $A(z)=mz$ and $W(z)=1/(1-mz)$, respectively.

Tuples or repetitions of letters and words make an appearance frequently in
our arguments. We construct them as follows:
Let $\mathcal J_k=\mathcal A\circ\mathcal Z^k$ be the set of all $k$-tuples
of the same letter in $\mathcal A$ and let $\mathcal K_k=\mathcal W\circ\mathcal Z^k$ be the
set of all $k$-tuples of the same word over $\mathcal A$. Then $\mathcal J_k$ and $\mathcal K_k$
have associated OGFs $J_k(z)=mz^k$ and $\mathcal K_k=1/(1-mz^k)$.

%\subsection{Complex asymptotics}

The following theorem greatly simplifies the process of finding the
asymptotics of a sequence, given knowledge of its generating function.

\begin{thm}\label{asymptotic}\cite[p.~258]{flajolet}
Let $f(z)$ be a function meromorphic at all points of the closed disc
$|z|\le R$, with poles at points $\alpha_1,\alpha_2,\ldots,\alpha_r$.
Assume that $f(z)$ is analytic at all points of $|z|=R$ and at $z=0$.
Then there exist $r$ polynomials $\{P_j\}_{j=1}^r$ such that
$$f_n:=[z^n]f(z)=\sum_{j=1}^rP_j(n)\alpha_j^{-n}+O(R^{-n}).$$
Furthermore the degree of $P_j$ is equal to the order of the pole of $f$
at $\alpha_j$ minus one.
\end{thm}

The following theorem formalizes our use of the probabilistic method.

\begin{thm}\label{expectation}\cite[p.~18]{probabilistic}
Let $(\Omega,\mathcal F,P)$ be a probability space and
$X:\Omega\rightarrow\mathbb R$ be a real-valued random variable, i.e. such that for all $x\in\mathbb R$,
$\{\omega\in\Omega:X(\omega)\le x\}\in\mathcal F$. Let
$$E[X]:=\int_{\Omega}\!X(\omega)\,\mathrm dP(\omega)$$
denote the mathematical expectation of $X$. If $E[X]<\infty$,
then for some $\omega\in\Omega$, $X(\omega)\le E[X]$.
\end{thm}
%\begin{proof}
%Suppose otherwise. Then for all $\omega\in\Omega$, $X(\omega)>E[X]$
%and $$E[X]=\int_{\Omega}\!X(\omega)\,\mathrm dP(\omega)
%>\int_{\Omega}\!E[X]\,\mathrm dP(\omega) 
%=E[X]\int_{\Omega}\!\,\mathrm dP(\omega) 
%=E[X]\cdot 1=E[X], %\qed
%$$
%a contradiction.
%\end{proof}

%Van der Waerden's theorem \cite{waerden} states that for arbitrary positive
%integers $m$ and $k$, there exists some number $n$ such that, if the integers
%$1,\ldots,n$ are colored, each with one of $m$ different colors, then the
%colored integers contain a monochromatic arithmetic progression of length $k$.
%The {\em van der Waerden number} $W(m,k)$ is the least such $n$. Applying Theorem~\ref{expectation} to Example~\ref{vdw} when
%$\widehat{\Omega}_n<1$, we get the following lower bound for
%van der Waerden numbers:
%$$W(m,k)\ge(1+o(1))\sqrt{2(k-1)m^{k-1}}.$$

Straightforward applications of the probabilistic method often give crude
results, as demonstrated below; nevertheless, they still provide important
qualitative information.

Define the $k$th {\em Ramsey number}
$R(k)$ to be the minimal value of $n$ in the
statement of Ramsey's theorem, Theorem~\ref{ramsey}, for a given value of $k$.

\begin{thm}\cite{ramsey}\label{ramsey}
For every positive integer $k$ there is a positive integer $n$, such that if
the edges of the complete graph on $n$ vertices are all colored either red
or blue, then there must be $k$ vertices such that all edges joining them have
the same color.
\end{thm}

In 1947, Erd\H{o}s proved the following result. %An outline of the proof is
%given by Gowers in \cite{gowers}; we expand on it and fill in the
%details.

\begin{thm}
$R(k)\ge 2^{k/2}$ for all $k\ge 2$.
\end{thm}
%\begin{proof}
%It is easy to see that $R(2)=2$, so we consider only the cases where $k\ge 3$.
%Let $G$ be a complete graph, and let $N$ be the number of vertices in $G$.
%Color each edge of $G$ red with probability $1/2$ and blue with probability
%$1/2$, and make all the choices independent. Consider any $k$-subset
%$S=\{x_1,\ldots,x_k\}$ of the vertices in $G$. The probability that
%$x_i$ is joined to $x_j$ by a red edge for every pair of distinct vertices
%in $S$ is $2^{-\binom{k}{2}}$, and so is the probability that $x_i$ is joined
%to $x_j$ by a blue edge for every pair of distinct vertices in $S$. Therefore,
%the expected number of sets of $k$ vertices all joined by edges of the same
%color is $2^{1-\binom{k}{2}}\binom{N}{k}$. If this is less than one, then
%it must be possible for there to be no such sets of $k$ vertices. This
%happens when $N=2^{k/2}$ since 
%$$\frac{2^{k/2}\cdots (2^{k/2}-k+1)}{k!}<\frac{2^{k^2/2}}{6\cdot 4^{k-3}}
%=2^{k^2/2-\log_2 6-2k+6}<2^{k^2/2-k/2-1}$$
%for $k\ge 3$ (it holds that $3k/2\ge 9/2>7-\log_2 6$) and the following
%inequalities are equivalent:
%\begin{align*}
%2^{1-\binom{k}{2}}\binom{2^{k/2}}{k} &< 1 \\
%2\binom{2^{k/2}}{k} &< 2^{\binom{k}{2}} \\
%2\left[\frac{2^{k/2}\cdots (2^{k/2}-k+1)}{k!}\right] &< 2^{k(k-1)/2} \\
%\left[\frac{2^{k/2}\cdots (2^{k/2}-k+1)}{k!}\right] &< 2^{(k^2-k-2)/2}. %\qed
%\end{align*}
%\end{proof}

Erd\H{o}s' lower bound, exponential with base $\sqrt{2}$, is rough (and
so are all lower bounds on $R(k)$ proven since then) because the best known
upper bounds are exponential with base $4$. In fact, some of the major
open problems in combinatorics, according to Gowers \cite{gowers}, are the following:
Does there exist a constant $a>\sqrt{2}$ such that $R(k)\ge a^k$ for all
sufficiently large $k$?
Does there exist a constant $b<4$ such that $R(k)\le b^k$ for all
sufficiently large $k$?
Although Erd\H{o}s' lower bound for $R(k)$ is crude, it tells us valuable
information about $R(k)$, namely that it grows at least exponentially.

\section{Pattern occurrence statistics: the full word case}

We calculate pattern occurrence statistics and prove results about the Ramsey theory of unavoidable patterns in the
full word case. We calculate the mean number of occurrences of a pattern in a full word of a
given length and use that statistic to establish a lower bound on Ramsey lengths. 
We do so for the nonabelian case in Section~4.1 and the abelian case in Section~4.2.  

\subsection{The nonabelian case}

Theorem~\ref{full} together with Corollary~\ref{corfull} answer the basic question as to when a full word can avoid a given pattern.

\begin{thm}\label{full}
Suppose that a pattern $p$ uses $r$ distinct variables, where the $j$th
variable occurs $k_j\ge 1$ times. Without loss of generality, let
$k=|p|=k_1+\cdots+k_r$ and $1=k_1=\cdots=k_s<k_{s+1}\le\cdots\le k_r$.
Then the mean number of occurrences of $p$ in a full word of length $n$ 
over an alphabet of $m$ letters is
$$\widehat{\Omega}_n
\sim\frac{1}{\prod_{j=s+1}^r(m^{k_j-1}-1)}\frac{n^{s+1}}{(s+1)!}.$$
\end{thm}
\begin{proof}
For the {\em mean} number of occurrences of a pattern $p$,
calculations similar to those employed for the number of occurrences of a 
word \cite[p.~61]{flajolet} can be based on regular specifications. 
Each occurrence of $p$ consists of a concatenation of nonempty words
(represented by $\mathcal{W}\setminus \{\varepsilon\}=\mathrm{SEQ}(\mathcal A)\setminus \{\varepsilon\}$) repeated
$k_j$ times for the $j$th variable, surrounded by arbitrary sequences of 
letters. Thus all the occurrences of $p$ as a factor are described by
$$\widehat{\mathcal O}
=\text{SEQ}(\mathcal A)\times\prod_{j=1}^{r}[(\mathcal W\backslash\{\varepsilon\})
\circ\mathcal Z^{k_j}]\times\text{SEQ}(\mathcal A),$$
so we get
\begin{align}
\widehat{O}(z)
&=\frac{1}{(1-mz)^2}\prod_{j=1}^r\left(\frac{1}{1-mz^{k_j}}-1\right) \notag \\
&=\frac{m^rz^k}{(1-mz)^{2+s}}\prod_{j=s+1}^r\frac{1}{1-mz^{k_j}}. \label{ogf}
\end{align}
We have a pole of order $2+s$ at $z=1/m$, and poles at
the $k_j$ different $k_j$th roots of $1/m$ for $k_j\ge 2$
(which have modulus greater than $1/m$). By Theorem \ref{asymptotic},
we know that for any $R>1$, there exist polynomials
$P_1,P_{s+1},\ldots,P_r$ such that
$$[z^n]\widehat{O}(z)=P_1(n)m^n+\sum_{j=s+1}^rP_j(n)m^{n/k_j}+O(R^{-n}),$$
where the degree of $P_1$ is $s+1$. For an asymptotic equivalent of
$[z^n]\widehat{O}(z)$, only the pole at $z=1/m$ needs to be considered
because it is closest to the origin and corresponds to the fastest
exponential growth; it is the dominant singularity. We plug in
$z=1/m$ in Equation~(\ref{ogf}) for the nonsingular portion to obtain
the first-order asymptotics of the OGF near $z=1/m$:
\begin{align*}
\widehat{O}(z)
&\sim\frac{m^{r-k}}{\prod_{j=s+1}^r(1-m^{1-k_j})}\frac{1}{(1-mz)^{2+s}} \\
&=\frac{1}{\prod_{j=s+1}^r(m^{k_j-1}-1)}\frac{1}{(1-mz)^{2+s}},
\end{align*}
which correspond to the first-order asymptotics of the associated sequence,
\begin{align*}
[z^n]\widehat{O}(z)
&\sim \frac{1}{\prod_{j=s+1}^r(m^{k_j-1}-1)}\binom{n+s+1}{s+1}m^n \\
&\sim \frac{1}{\prod_{j=s+1}^r(m^{k_j-1}-1)}\frac{n^{s+1}m^n}{(s+1)!}.
\end{align*}
Therefore, the mean number of occurrences of a pattern $p$
in a word of length $n$ over an alphabet of $m$ letters is
$$\widehat{\Omega}_n
\sim\frac{1}{\prod_{j=s+1}^r(m^{k_j-1}-1)}\frac{n^{s+1}}{(s+1)!}. %\qed
$$
\end{proof}

To illustrate Theorem~\ref{full}, consider the pattern $p=abacaba$, where $r=3$, $s=1$, and where $k_1=1$, $k_2=2$,
and $k_3=4$ denote, respectively, the number of occurrences of $c$, $b$, and $a$ in $p$. Substituting these variables, $m=12$, and $n=100$, we find
that $$\widehat{\Omega}_{100}\approx 0.26319\cdots.$$

When $\widehat{\Omega}_n<1$, we may apply Theorem \ref{expectation},
so in that case we obtain the following corollary.

\begin{cor}\label{corfull}
Suppose that a pattern $p$ uses $r$ distinct variables, where the $j$th
variable occurs $k_j\ge 1$ times. Without loss of generality, let
$|p|=k_1+\cdots+k_r$ and $1=k_1=\cdots=k_s<k_{s+1}\le\cdots\le k_r$. If
$$n<(1+o(1))\left[(s+1)!\prod_{j=s+1}^r(m^{k_j-1}-1)\right]^{\frac{1}{s+1}},$$
there is a word of length $n$ over an alphabet of
$m$ letters that avoids $p$.
\end{cor}

This result is rather crude. It says that the maximum length
of words avoiding $aa$ is at least $m-2$, but $m$ is the cardinality of the
alphabet. Nevertheless, it says that, as a variable is repeated, the maximum
length of words avoiding the associated pattern grows at least exponentially.
For example, the maximum length of words avoiding the pattern $a^k$ is at
least $m^{k-1}-2$. The maximum length of words avoiding the Zimin pattern
$Z_i$, as defined in Equation~(\ref{eq21}), is at least
$$-1+\sqrt{2\prod_{j=1}^{i-1}(m^{2^j-1}-1)}.$$ 
Since $Z_i$ is unavoidable,
$Z_i$ has an associated Ramsey length $L(m,Z_i)$, and
we get 
\begin{equation}
L(m,Z_i)\ge(1+o(1))\sqrt{2\prod_{j=1}^{i-1}(m^{2^j-1}-1)}.\label{eq41}
\end{equation}
Substituting for example $m=12$ and $i=3$, we find that $$L(12,Z_3)\ge 194.92\cdots.$$

According to \cite[p.~101]{lothaire},
$$L(m,Z_{i})\le m^{L(m,Z_{i-1})}[L(m,Z_{i-1})+1]+L(m,Z_{i-1})$$
and $L(m,Z_2)=2m+1$. The best possible nonrecursive
upper bound for $L(m,Z_i)$ deducible from our results is cumbersome
to write, so we settle for a more convenient but less precise one, in terms of
Knuth's up-arrow notation. 

\begin{thm}
\label{upperboundzimin}
For $m\ge 2$ and $i\ge 2$, 
\begin{equation}
L(m,Z_i)<m\uparrow\uparrow(2i-1).\label{eq42}
\end{equation}
\end{thm}
\begin{proof}
Since $m\ge 2$, $L(m,Z_2)=2m+1<m^{m^m}=m\uparrow\uparrow 3$ establishes
our base case. For the inductive step, assume that for some $i\ge 2$,
$$L(m,Z_i)<m\uparrow\uparrow(2i-1).$$ As stated earlier,
$$L(m,Z_{i+1})\le m^{L(m,Z_i)}[L(m,Z_i)+1]+L(m,Z_i),$$
so
\begin{align*}
L(m,Z_{i+1}) &< [m^{L(m,Z_i)}+1][L(m,Z_i)+1] \\
&<[m^{m\uparrow\uparrow(2i-1)}+1][m\uparrow\uparrow(2i-1)+1] \\
&=[m\uparrow\uparrow(2i)+1][m\uparrow\uparrow(2i-1)+1] \\
&<[m\uparrow\uparrow(2i)+1][m\uparrow\uparrow(2i)] \\
&<[m\uparrow\uparrow(2i)]^3 \\
&=m^{3 m\uparrow\uparrow(2i-1)} \\
&<m^{m\uparrow\uparrow(2i)} \\
&=m\uparrow\uparrow(2i+1),
\end{align*}
and our induction is complete. %\qed
\end{proof}

Our derived upper bound for $L(m,Z_i)$ in Equation~(\ref{eq42}), which uses tetration,
is vastly greater than our derived lower bound for $L(m,Z_i)$ in Equation~(\ref{eq41}), which
uses repeated squaring. Nevertheless, we have established
concrete upper and lower bounds for $L(m,Z_i)$. 

\subsection{The abelian case}

Since it is not obvious whether the generating
function may be analytically continued beyond its radius of convergence,
we treat it as though it is {\em lacunary}, i.e. not analytically continuable,
and we use techniques from \cite{darboux} to calculate asymptotics.

Theorem~\ref{abelian} together with Corollary~\ref{corabelian} answer the basic question as to when a full word can avoid a given pattern in the abelian sense.

\begin{thm}\label{abelian}
Suppose that a pattern $p$ uses $r$ distinct variables, where the $j$th
variable occurs $k_j\ge 1$ times. Without loss of generality, let
$k=|p|=k_1+\cdots+k_r$ and $1=k_1=\cdots=k_s<k_{s+1}\le\cdots\le k_r$.
Then the mean number of occurrences of $p$ in the abelian sense in
a word of length $n$ over an alphabet of $m\ge 4$ letters is
$$\widehat{\Omega}_n\sim\frac{n^{s+1}}{(s+1)!}\prod_{j=s+1}^r\left[
\sum_{\ell=1}^{\infty}\frac{1}{m^{k_j\ell}}\sum_{i_1+\cdots+i_m=\ell}
\binom{\ell}{i_1,\ldots,i_m}^{k_j}\right].$$
\end{thm}

\begin{proof}
For the {\em mean} number of occurrences of a pattern $p$ in the abelian sense,
calculations similar to those employed for the number of occurrences of a
pattern $p$ in the nonabelian sense can be based on regular specifications.
Each occurrence of $p$ consists of a concatenation of nonempty words repeated
$k_j$ times for the $j$th variable surrounded by arbitrary sequences of 
letters, with the additional $k_j-1$ instances of each substituted word
being allowed to permute their letters. Thus all the occurrences of $p$
as a factor in the abelian sense are described by
$$\widehat{\mathcal O}=\text{SEQ}(\mathcal A)\times\prod_{j=1}^{r}
\left(\sum_{w\in\mathcal W\backslash\{\varepsilon\}}|\Per(w)|^{k_j-1}
\mathcal Z^{k_j|w|}\right)\times\text{SEQ}(\mathcal A),$$
where $\Per(w)$ denotes the set of distinct permutations of the word $w$. So we get
\begin{align*}
\widehat{O}(z) &= \frac{1}{(1-mz)^2}\prod_{j=1}^r\left(
\sum_{\ell=1}^{\infty}z^{k_j\ell}\sum_{i_1+\cdots+i_m=\ell}
\binom{\ell}{i_1,\ldots,i_m}^{k_j}\right) \\
&= \frac{m^sz^s}{(1-mz)^{2+s}}\prod_{j=s+1}^r\left(
\sum_{\ell=1}^{\infty}z^{k_j\ell}\sum_{i_1+\cdots+i_m=\ell}
\binom{\ell}{i_1,\ldots,i_m}^{k_j}\right).
\end{align*}

Note that for $s+1\le j\le r$, $k_j\ge 2$. For $\ell\ge 2$,
$$\sum_{i_1+\cdots+i_m=\ell}
\binom{\ell}{i_1,\ldots,i_m}^{k_j}>\sum_{i_1+\cdots+i_m=\ell}
\binom{\ell}{i_1,\ldots,i_m}=m^{\ell},$$
establishing that
$$\sum_{\ell=1}^{\infty}z^{k_j\ell}\sum_{i_1+\cdots+i_m=\ell}
\binom{\ell}{i_1,\ldots,i_m}^{k_j}$$
has radius of convergence $R<1/m^{1/k_j}$.
Since $$\binom{\ell}{i_1,\ldots,i_m}<\sum_{i_1+\cdots+i_m=\ell}
\binom{\ell}{i_1,\ldots,i_m}=m^{\ell},$$
we can apply \cite[Theorem~4]{rs} and get
\begin{align*}
\frac{1}{m^{k_j\ell}}\sum_{i_1+\cdots+i_m=\ell}
\binom{\ell}{i_1,\ldots,i_m}^{k_j}
&=\sum_{i_1+\cdots+i_m=\ell}
\left[\frac{\binom{\ell}{i_1,\ldots,i_m}}{m^{\ell}}\right]^{k_j} \\
&\le\sum_{i_1+\cdots+i_m=\ell}
\left[\frac{\binom{\ell}{i_1,\ldots,i_m}}{m^{\ell}}\right]^2 \\
&=\frac{1}{m^{2\ell}}
\sum_{i_1+\cdots+i_m=\ell}\binom{\ell}{i_1,\ldots,i_m}^2 \\
&\sim m^{m/2}(4\pi\ell)^{(1-m)/2}.
\end{align*}
Consequently, since $m\ge 4$, the Riemann zeta function $\zeta(\frac{m-1}{2})=\displaystyle\sum_{\ell=1}^{\infty}\frac{1}{\ell^{\frac{m-1}{2}}}$ converges, and
\begin{align*}
\sum_{\ell=1}^{\infty}\frac{1}{m^{k_j\ell}}\sum_{i_1+\cdots+i_m=\ell}
\binom{\ell}{i_1,\ldots,i_m}^{k_j}
&\le\sum_{\ell=1}^{\infty}(1+o(1))m^{m/2}(4\pi\ell)^{(1-m)/2} \\
&\sim m^{m/2}(4\pi)^{(1-m)/2}\zeta\left(\frac{m-1}{2}\right) <\infty.
\end{align*}
Thus we establish that
$$\sum_{\ell=1}^{\infty}z^{k_j\ell}\sum_{i_1+\cdots+i_m=\ell}
\binom{\ell}{i_1,\ldots,i_m}^{k_j}$$
converges at all $z$ in the closed disc $|z|\le 1/m$, since
all the power series coefficients are nonnegative and
\begin{align*}
\left|\sum_{\ell=1}^{\infty}z^{k_j\ell}\sum_{i_1+\cdots+i_m=\ell}
\binom{\ell}{i_1,\ldots,i_m}^{k_j}\right|
&\le\sum_{\ell=1}^{\infty}|z|^{k_j\ell}\sum_{i_1+\cdots+i_m=\ell}
\binom{\ell}{i_1,\ldots,i_m}^{k_j} \\
&\le\sum_{\ell=1}^{\infty}\frac{1}{m^{k_j\ell}}\sum_{i_1+\cdots+i_m=\ell}
\binom{\ell}{i_1,\ldots,i_m}^{k_j}<\infty.
\end{align*}
In fact, $\displaystyle\sum_{\ell=1}^{\infty}z^{2\ell}\sum_{i_1+\cdots+i_m=\ell}
\binom{\ell}{i_1,\ldots,i_m}^2$ has radius of convergence $R=1/m$
since, by the Cauchy-Hadamard theorem, the radius of convergence satisfies
\begin{align*}
\frac{1}{R} &= \limsup_{\ell\rightarrow\infty}
\left[\sum_{i_1+\cdots+i_m=\ell}\binom{\ell}{i_1,\ldots,i_m}^2
\right]^{\frac{1}{2\ell}} \\
&= \limsup_{\ell\rightarrow\infty}
\left[m^{2\ell+\frac{m}{2}}(4\pi\ell)^{(1-m)/2}
\right]^{\frac{1}{2\ell}} \\
&= \lim_{\ell\rightarrow\infty}
\left[m^{1+\frac{m}{4\ell}}(4\pi\ell)^{\frac{1-m}{4\ell}}\right] \\
&= m.
\end{align*}
We may factor $\widehat{O}(z)$ as $\widehat{O}(z)=P(mz)\cdot Q(mz)$, where
$$P(z)=\dfrac{1}{(1-z)^{2+s}}$$ and $$Q(z)=z^s\displaystyle\prod_{j=s+1}^r\left[
\sum_{\ell=1}^{\infty}\left(\frac{z}{m}\right)^{k_j\ell}
\sum_{i_1+\cdots+i_m=\ell}\binom{\ell}{i_1,\ldots,i_m}^{k_j}\right].$$
%Since $Q(z)$ is analytic in $|z|<1$ and converges at all points
%on the unit disc, 
%it is $\mathcal C^0$-smooth on the unit circle by the
%following argument: for any real $\epsilon>0$, there exists
%$\delta_1>0$ and $\delta_2>0$ such that for any two
%points $z_1,z_3$ on the unit circle and point $z_2$ in the open unit disc,
%$$|Q(z_1)-Q(z_2)|<\epsilon/2 \text{ whenever }|z_2-z_1|<\delta_1$$
%and
%$$|Q(z_2)-Q(z_3)|<\epsilon/2 \text{ whenever }|z_3-z_2|<\delta_2.$$
%If we require $z_1$ and $z_3$ to be distinct and set $z_2=(z_1+z_3)/2$, then 
%by the triangle inequality we get $$|Q(z_1)-Q(z_3)|<\epsilon,$$
%establishing the continuity of $Q(z)$ on $|z|\le 1$. Thus, by Definition~\ref{cssmooth},
%$Q(z)$ is $\mathcal C^0$-smooth on the unit circle.
Note that $Q(z)$ is analytic in $|z|<1$ and converges at all points
on the unit disc. Also note that $Q(z)$ is $\mathcal C^\infty$-smooth on the unit circle; differentiating the power series any number of times does not make it diverge. In particular, $Q(z)$ is $\mathcal C^{2+s}$-smooth on the unit circle.

Note that $P(z)$ is of global order $-2-s$ and is its own log-power expansion
of type $\mathcal O^t$ relative to $W=\{1\}$, where $t=\infty$. Since
$t=\infty>u_0=\lfloor((2+s)+(-2-s))/2\rfloor \geq 0$,
the conditions of \cite[Theorem~1]{darboux} hold. Letting
$c_0=\lfloor((2+s)-(-2-s))/2\rfloor=2+s$,
we find that $$[z^n](P(z)\cdot Q(z))=[z^n](P(z)\cdot H(z))
+o(1),$$
where $H(z)$ is the Hermite interpolation polynomial such that all its
derivatives of order $0,\ldots, 1+s$ coincide with those of $Q(z)$
at $w=1$. Note that this implies that $$H(1)=\prod_{j=s+1}^r\left[
\sum_{\ell=1}^{\infty}\frac{1}{m^{k_j\ell}}\sum_{i_1+\cdots+i_m=\ell}
\binom{\ell}{i_1,\ldots,i_m}^{k_j}\right].$$
Scaling by a factor of $m$, we get
$$[z^n]\widehat{O}(z)=[z^n](P(mz)\cdot H(mz))
+o(1).$$
Since $H(z)$ is a polynomial, the only singularity of $P(mz)\cdot H(mz)$ is
$z=1/m$, so it dominates, and by Theorem \ref{asymptotic},
\begin{align*}
[z^n]\widehat{O}(z)
&\sim\binom{n+s+1}{s+1}m^n\prod_{j=s+1}^r\left[
\sum_{\ell=1}^{\infty}\frac{1}{m^{k_j\ell}}\sum_{i_1+\cdots+i_m=\ell}
\binom{\ell}{i_1,\ldots,i_m}^{k_j}\right] \\
&\sim\frac{n^{s+1}m^n}{(s+1)!}\prod_{j=s+1}^r\left[
\sum_{\ell=1}^{\infty}\frac{1}{m^{k_j\ell}}\sum_{i_1+\cdots+i_m=\ell}
\binom{\ell}{i_1,\ldots,i_m}^{k_j}\right].
\end{align*}
Therefore, the mean number of occurrences of a pattern
$p$ in the abelian sense in a word of length $n$ over an alphabet of $m\ge 4$ letters is
$$\widehat{\Omega}_n\sim\frac{n^{s+1}}{(s+1)!}\prod_{j=s+1}^r\left[
\sum_{\ell=1}^{\infty}\frac{1}{m^{k_j\ell}}\sum_{i_1+\cdots+i_m=\ell}
\binom{\ell}{i_1,\ldots,i_m}^{k_j}\right]. %\qed
$$
\end{proof}

To illustrate Theorem~\ref{abelian}, consider the pattern $p=aba$, where $r=2$, $s=1$, $k_1=1$, and $k_2=2$.
Substituting these variables, $m=12$, and $n=100$, and applying
\cite[Theorem~4]{rs} we find
that $$\widehat{\Omega}_{100}\approx\frac{100^2\cdot 12^{12/2}(4\pi)^{-11/2}}{2}
\zeta\left(\frac{11}{2}\right)\approx 13778.87\cdots.$$

When $\widehat{\Omega}_n<1$, we may apply Theorem \ref{expectation},
so in that case we obtain the following corollary.

\begin{cor}\label{corabelian}
Suppose that a pattern $p$ uses $r$ distinct variables, where the $j$th
variable occurs $k_j\ge 1$ times. Without loss of generality, let
$|p|=k_1+\cdots+k_r$ and $1=k_1=\cdots=k_s<k_{s+1}\le\cdots\le k_r$. For $$n<(1+o(1))\left\{(s+1)!\prod_{j=s+1}^r\left[
\sum_{\ell=1}^{\infty}\frac{1}{m^{k_j\ell}}\sum_{i_1+\cdots+i_m=\ell}
\binom{\ell}{i_1,\ldots,i_m}^{k_j}\right]^{-1}\right\}^{\frac{1}{s+1}},$$
there is a word of length $n$ over an alphabet of $m\ge 4$ letters that
avoids $p$ in the abelian sense.
\end{cor}

The upper bound for $L(m,Z_i)$ in Theorem~\ref{upperboundzimin} also applies to $L_{\mathrm{ab}}(m,Z_i)$. For a lower bound, we get the following.

\begin{cor}
Let $m\ge 4$. Then
$$L_{\mathrm{ab}}(m,Z_i)
\ge(1+o(1))\sqrt{2\prod_{j=1}^{i-1}\left[\sum_{\ell=1}^{\infty}
\frac{1}{m^{2^j\ell}}\sum_{i_1+\cdots+i_m=\ell}
\binom{\ell}{i_1,\ldots,i_m}^{2^j}\right]^{-1}}.$$
\end{cor}

\section{Pattern occurrence statistics: the partial word case}

Next, we investigate the case of patterns in partial words. As in the case
of full words, we calculate the mean number of pattern occurrences. First,
we take the average over all partial words of a given length. Then we average
over all strictly partial words of a given length, and finally, we take the average over all
partial words of a given length with a given hole density. The last of these
statistics, gotten through the calculation of bivariate asymptotics, allows
us to prove a lower bound on partial Ramsey lengths. 

\subsection{Mean over all partial words of a given length}

The following lemma, which can be proved by induction on $k$, will help us compare the distances of poles of
the generating function from the origin and establish one of them
as the dominant singularity.

\begin{lma}\label{mk}
If $m\ge 2$ and $k\ge 2$, then $m2^k-m+1<(m+1)^k$.
\end{lma}
%\begin{proof}
%Fixing $m\ge 2$, we induct on $k$. The base case $k=2$ clearly holds.
%For the inductive step, assume that the lemma holds for $k=j$, i.e.
%$m2^j-m+1<(m+1)^j$. Since $m2^j < m(m+1)^j$, we may add the inequalities
%together and get $m2^{j+1}-m+1<(m+1)^{j+1}$, so the lemma holds for
%$k=j+1$, completing the induction. %\qed
%\end{proof}

Theorem~\ref{partial} together with Corollary~\ref{corpartial} answer the basic question as to when a partial word can avoid a given pattern.

\begin{thm}\label{partial}
Suppose that a pattern $p$ uses $r$ distinct variables, where the $j$th variable
occurs $k_j\ge 1$ times. Without loss of generality, let $k=|p|=k_1+\cdots+k_r$
and $1=k_1=\cdots=k_s<k_{s+1}\le\cdots\le k_r$. Then the mean number of occurrences
of $p$ in a partial word of length $n$ over an alphabet of $m$ letters is
$$\widehat{\Omega}_n\sim\frac{n^{s+1}}{(s+1)!}\prod_{j=s+1}^r
\frac{m2^{k_j}-m+1}{(m+1)^{k_j}-(m2^{k_j}-m+1)}.$$
\end{thm}

\begin{proof}
For the {\em mean} number of occurrences of a pattern $p$ in
a partial word of length $n$, calculations similar to those
employed for the number of occurrences of a pattern $p$ in a full word
of length $n$ can be based on regular specifications.
Each occurrence of $p$ consists of a concatenation of nonempty full words repeated
$k_j$ times for the $j$th variable surrounded by arbitrary sequences of letters and hole characters, with the option of having some letters in the substituted
words be replaced by $\diamond$'s. When a letter in a word is replaced in
every instance by $\diamond$, that letter practically no longer exists, and we
treat it like $\diamond$. Thus all the occurrences of $p$ as a factor are
described by
$$\widehat{\mathcal O}=\text{SEQ}(\mathcal A+\{\diamond\})
\times\prod_{j=1}^{r}\text{SEQ}(\{\diamond^{k_j}\}
+\mathcal A\circ[(\mathcal Z+\{\diamond\})^{k_j}\backslash\{\diamond^{k_j}\}])
\backslash\{\varepsilon\}\times\text{SEQ}(\mathcal A+\{\diamond\}),$$ so we get
\begin{align*}
\widehat{O}(z)
&=\frac{1}{[1-(m+1)z]^2}
\prod_{j=1}^r\left(\frac{1}{1-z^{k_j}-m(2^{k_j}-1)z^{k_j}}-1\right) \\
&=\frac{(m+1)^sz^s}{[1-(m+1)z]^{2+s}}
\prod_{j=s+1}^r\frac{(m2^{k_j}-m+1)z^{k_j}}{1-(m2^{k_j}-m+1)z^{k_j}} \\
&=\frac{(m+1)^sz^k}{[1-(m+1)z]^{2+s}}
\prod_{j=s+1}^r\frac{m2^{k_j}-m+1}{1-(m2^{k_j}-m+1)z^{k_j}}.
\end{align*}
We have a pole of order $2+s$ at $z=1/(m+1)$, and poles at
the $k_j$ different $k_j$th roots of $1/(m2^{k_j}-m+1)$ for $k_j\ge 2$.
Those poles have modulus greater than $1/(m+1)$ by Lemma \ref{mk}.
The singularity at $z=1/(m+1)$ dominates
because it is closest to the origin, so by Theorem \ref{asymptotic},
\begin{align*}
\widehat{O}(z)
&\sim \frac{(m+1)^{s-k}}{[1-(m+1)z]^{2+s}}\prod_{j=s+1}^r
\frac{m2^{k_j}-m+1}{1-(m2^{k_j}-m+1)/(m+1)^{k_j}} \\
&= \frac{1}{[1-(m+1)z]^{2+s}}\prod_{j=s+1}^r
\frac{m2^{k_j}-m+1}{(m+1)^{k_j}-(m2^{k_j}-m+1)}.
\end{align*}
Taking the coefficient of $z^n$ in the Taylor expansion, we get
\begin{align*}
[z^n]\widehat{O}(z)
&\sim\binom{n+s+1}{s+1}(m+1)^n\prod_{j=s+1}^r
\frac{m2^{k_j}-m+1}{(m+1)^{k_j}-(m2^{k_j}-m+1)} \\
&\sim\frac{n^{s+1}(m+1)^n}{(s+1)!}\prod_{j=s+1}^r
\frac{m2^{k_j}-m+1}{(m+1)^{k_j}-(m2^{k_j}-m+1)}.
\end{align*}
Therefore, the mean number of occurrences of a pattern $p$
in a partial word of length $n$ over an alphabet of $m$ letters is 
$$\widehat{\Omega}_n
\sim\frac{n^{s+1}}{(s+1)!}\prod_{j=s+1}^r
\frac{m2^{k_j}-m+1}{(m+1)^{k_j}-(m2^{k_j}-m+1)}. %\qed
$$
\end{proof}

To illustrate Theorem~\ref{partial}, consider the pattern $p=abacaba$, where $r=3$, $s=1$, $k_1=1$, $k_2=2$,
and $k_3=4$. Substituting these variables, $m=12$, and $n=100$, we find
that $$\widehat{\Omega}_{100}\approx 8.9384\cdots.$$

When $\widehat{\Omega}_n<1$, we may apply Theorem \ref{expectation},
so in that case we obtain the following corollary.

\begin{cor}\label{corpartial}
Suppose that a pattern $p$ uses $r$ distinct variables, where the $j$th variable
occurs $k_j\ge 1$ times. Without loss of generality, let $k=|p|=k_1+\cdots+k_r$
and $1=k_1=\cdots=k_s<k_{s+1}\le\cdots\le k_r$. If $$n<(1+o(1))\left[(s+1)!\prod_{j=s+1}^r
\left(\frac{(m+1)^{k_j}}{m2^{k_j}-m+1}-1\right)\right]^{\frac{1}{s+1}},$$
there is a partial word of length $n$ over an alphabet of $m$ letters
that avoids $p$.
\end{cor}

\subsection{Mean over all stricly partial words of a given length}

When a partial word avoids a pattern, all of its completions also do,
so the above corollary is weaker than the corresponding
one for full words. What we would want to do is to require that the
partial words be strictly partial, i.e. have at least one hole. 
%Then, supposedly, we would get meaningful results. Unfortunately, 
However, we get the same asymptotics.

\begin{thm}
\label{striclypartial}
Suppose that a pattern $p$ uses $r$ distinct variables, where the $j$th
variable occurs $k_j\ge 1$ times. Without loss of generality, let
$k=|p|=k_1+\cdots+k_r$ and $1=k_1=\cdots=k_s<k_{s+1}\le\cdots\le k_r$. The
mean number of occurrences of $p$ in a strictly partial word of length $n$ over an alphabet of $m$ letters is
$$\widehat{\Omega}_n\sim\frac{n^{s+1}}{(s+1)!}\prod_{j=s+1}^r
\frac{m2^{k_j}-m+1}{(m+1)^{k_j}-(m2^{k_j}-m+1)}.$$
\end{thm}
\begin{proof}
As we know from Theorem \ref{partial}, the total number of occurrences of a
pattern $p$ in partial words of length $n$ over an alphabet of $m$ letters is
$$\sim\frac{n^{s+1}(m+1)^n}{(s+1)!}\prod_{j=s+1}^r
\frac{m2^{k_j}-m+1}{(m+1)^{k_j}-(m2^{k_j}-m+1)}.$$
According to Theorem \ref{full}, the total number of occurrences of a pattern
$p$ in full words of length $n$ over an alphabet of $m$ letters is
$$\sim\frac{1}{\prod_{j=s+1}^r(m^{k_j-1}-1)}\frac{n^{s+1}m^n}{(s+1)!}.$$
Subtracting the two quantities, we get the total number of occurrences of $p$
in strictly partial words of length $n$ over an alphabet of $m$ letters:
$$\sim\frac{n^{s+1}}{(s+1)!}\left[(m+1)^n\prod_{j=s+1}^r
\frac{m2^{k_j}-m+1}{(m+1)^{k_j}-(m2^{k_j}-m+1)}
-\frac{m^n}{\prod_{j=s+1}^r(m^{k_j-1}-1)}\right].$$
Dividing by $(m+1)^n-m^n$, the total number of strictly partial words
of length $n$ over an alphabet of $m$ letters, we find that the
mean number of occurrences of $p$ in a strictly partial word of length $n$
over an alphabet of $m$ letters is
$$\sim\frac{n^{s+1}/(s+1)!}{1-[m/(m+1)]^n}\left\{\prod_{j=s+1}^r
\frac{m2^{k_j}-m+1}{(m+1)^{k_j}-(m2^{k_j}-m+1)}
-\frac{[m/(m+1)]^n}{\prod_{j=s+1}^r(m^{k_j-1}-1)}\right\}.$$
However, notice that
$$\lim_{n\rightarrow\infty}\left(\frac{m}{m+1}\right)^n=0.$$
We find that the asymptotics are equivalent to
$$\widehat{\Omega}_n\sim\frac{n^{s+1}}{(s+1)!}\prod_{j=s+1}^r
\frac{m2^{k_j}-m+1}{(m+1)^{k_j}-(m2^{k_j}-m+1)},$$
the mean number of occurrences of $p$ in a loosely partial word of length $n$
over an alphabet of $m$ letters. %\qed
\end{proof}

To illustrate Theorem~\ref{striclypartial}, consider the pattern $p=abacaba$, where $r=3$, $s=1$, $k_1=1$, $k_2=2$,
and $k_3=4$. Substituting these variables, $m=12$, and $n=100$, we find
that $$\widehat{\Omega}_{100}\approx 8.9384\cdots.$$

When $\widehat{\Omega}_n<1$, we may apply Theorem \ref{expectation},
so in that case we obtain the following corollary.

\begin{cor}
Suppose that a pattern $p$ uses $r$ distinct variables, where the $j$th
variable occurs $k_j\ge 1$ times. Without loss of generality, let
$k=|p|=k_1+\cdots+k_r$ and $1=k_1=\cdots=k_s<k_{s+1}\le\cdots\le k_r$. 
If $$n<(1+o(1))\left[(s+1)!\prod_{j=s+1}^r
\left(\frac{(m+1)^{k_j}}{m2^{k_j}-m+1}-1\right)\right]^{\frac{1}{s+1}},$$
there is a strictly partial word of length
$n$ over an alphabet of $m$ letters that avoids $p$.
\end{cor}

\subsection{Mean over all partial words of a given length with a given hole density}

For all terms and notations not defined here, we refer the reader to the book of Pemantle and Wilson \cite{pemantle}.

Lemma~\ref{amoebas} will help us compare the distances of poles of
the bivariate version of the generating function from the origin and establish one of them
as the dominant singularity. %Figure~\ref{fig} illustrates it. 
%\begin{figure}
%\begin{center}
%\includegraphics[scale=.25]{amoeba1.png}
%\includegraphics[scale=.25]{amoeba2.png}
%\includegraphics[scale=.25]{amoeba3.png}

%\includegraphics[scale=.25]{amoeba4.png}
%\includegraphics[scale=.25]{amoeba5.png}

%\caption{Components $B_j$ of $\mathbb R^2\backslash\texttt{amoeba}(H_j)$
%containing a ray $(-\infty,b]\cdot(1,1)$ for $j\in[1..5]$}\label{fig}
%\end{center}
%\end{figure}

\begin{lma}\label{amoebas}
Let $H_1=1-(m+u)z$, and more generally let $H_j=1-[m(1+u)^j-mu^j+u^j]z^j$ for
integers $j\ge 2$. Let $B_1$ be the component of
$\mathbb R^2\backslash\texttt{amoeba}(H_1)$ containing a ray
$(-\infty,b]\cdot(1,1)$, and more generally let $B_j$ be the component of
$\mathbb R^2\backslash\texttt{amoeba}(H_j)$ containing a ray
$(-\infty,b]\cdot(1,1)$ for integers $j\ge 2$. Then
$$\partial B_1=\{(-\log(m+u),\log u):u\in(0,\infty)\},$$ $$\partial B_j
=\left\{\left(-\frac{1}{j}\log[m(1+u)^j-(m-1)u^j],\log u\right):u\in(0,\infty)
\right\}$$ for $j\ge 2$, and $B_1\subset B_j$ for all $j\ge 2$.
\end{lma}

\begin{proof}
First, note that $$\texttt{amoeba}(H_1)=\{(-\log|m+u|,\log|u|):u\in\mathbb C\}$$
and $$\texttt{amoeba}(H_j)
=\left\{\left(-\frac{1}{j}\log|m(1+u)^j-(m-1)u^j|,\log|u|\right)
:u\in\mathbb C\right\}$$
for integers $j \ge 2$.
Since the polynomial $m+u$ has all positive coefficients,
$$\log|m+u|\le\log(m+|u|)$$ and
$$\partial B_1=\{(-\log(m+u),\log u):u\in(0,\infty)\}.$$
More generally, the polynomial $m(1+u)^j-(m-1)u^j$ has all positive
coefficients, so $\log|m(1+u)^j-(m-1)u^j|\le\log[m(1+|u|)^j-(m-1)|u|^j]$ and
$$\partial B_j
=\left\{\left(-\frac{1}{j}\log[m(1+u)^j-(m-1)u^j],\log u\right):u\in(0,\infty)\right\}$$
for $j \ge 2$.
For any parameter $u\in(0,\infty)$, the corresponding points on $\partial B_1$
and $\partial B_j$ lie on the same horizontal line. However, the power means
inequality gives us
\begin{center}
$\sqrt[j]{\frac{(m+u)^j+(m-1)u^j}{m}} > \frac{(m+u)+(m-1)u}{m}=1+u$
\end{center}
which implies $$-\log(m+u) < -\frac{1}{j}\log[m(1+u)^j-(m-1)u^j],$$
so $\partial B_1$ lies strictly to the left of $\partial B_j$ and
$B_1\subset B_j$ for all $j\ge 2$. %\qed
\end{proof}

We now calculate the mean number of occurrences
of a pattern in a partial word with a given length and hole density.
This requires the construction of a bivariate generating function, where
a second variable, $u$, marks the number of $\diamond$'s in a partial word.

\begin{thm}
\label{partialdensity}
Suppose that a pattern $p$ uses $r$ distinct variables, where the $j$th variable
occurs $k_j\ge 1$ times. Without loss of generality, let $k=|p|=k_1+\cdots+k_r$
and $1=k_1=\cdots=k_s<k_{s+1}\le\cdots\le k_r$. Then the mean number of occurrences
of $p$ in a partial word of length $n$ with hole density
$d\in\mathbb Q\cap(0,1)$ over an alphabet of $m$ letters is
$$\widehat{\Omega}_{n,d}\sim\frac{n^{s+1}}{(s+1)!}
\prod_{j=s+1}^r\frac{[1+d(m-1)]^{k_j}-\left(1-\frac1m\right)(md)^{k_j}}
{m^{k_j-1}-[1+d(m-1)]^{k_j}+\left(1-\frac1m\right)(md)^{k_j}}.$$
\end{thm}
\begin{proof}
Marking each $\diamond$ with the variable $u$,
$$\widehat{\mathcal O}=\text{SEQ}(\mathcal A+\{\diamond\})
\times\prod_{j=1}^{r}\text{SEQ}(\{\diamond^{k_j}\}
+\mathcal A\circ[(\mathcal Z+\{\diamond\})^{k_j}\backslash\{\diamond^{k_j}\}])
\backslash\{\varepsilon\}\times\text{SEQ}(\mathcal A+\{\diamond\}),$$
becomes
\begin{align*}
\widehat{O}(z,u)
&=\frac{1}{(1-mz-uz)^2}\prod_{j=1}^r\left(
\frac{1}{1-u^{k_j}z^{k_j}-m[(1+u)^{k_j}-u^{k_j}]z^{k_j}}-1\right) \\
&=\frac{(m+u)^sz^k}{[1-(m+u)z]^{2+s}}
\prod_{j=s+1}^r\frac{m(1+u)^{k_j}-mu^{k_j}+u^{k_j}}
{1-[m(1+u)^{k_j}-mu^{k_j}+u^{k_j}]z^{k_j}}.
\end{align*}
For convenience, write
$G=(m+u)^sz^k\!\displaystyle\prod_{j=s+1}^r\![m(1+u)^{k_j}-mu^{k_j}+u^{k_j}]$
and $H=H_1^{2+s}\!\displaystyle\prod_{j=s+1}^r\!H_{k_j}$, where
$H_1=1-(m+u)z$ and $H_j=1-[m(1+u)^j-mu^j+u^j]z^j$, so that $\widehat{O}(z,u)=\dfrac{G}{H}$. Set $F=\dfrac{G}{H}$.

Note that $F$ is singular where $H$ is zero, i.e. on the singular variety
$$\mathcal V:=\mathcal V_H
=\left\{(z,u)\in\mathbb C^2:H_1^{2+s}\!\displaystyle\prod_{j=s+1}^r\!H_{k_j}
=0\right\}
=\mathcal V_1\cup\left(\bigcup_{j=s+1}^r\mathcal V_{k_j}\right)
$$
where we define $\mathcal V_1:=\{(z,u)\in\mathbb C^2:1-(m+u)z=0\}$
and $\mathcal V_j:=\{(z,u)\in\mathbb C^2:1-[m(1+u)^j-mu^j+u^j]z^j=0\}$.
Taking the log-modulus gives us the associated amoebas,
$\texttt{amoeba}(H):=\texttt{amoeba}(H_1)\cup\left(\bigcup_{j=s+1}^r
\texttt{amoeba}(H_{k_j})\right)$, where
$$\texttt{amoeba}(H_1)=\{(-\log|m+u|,\log|u|):u\in\mathbb C\}$$
and $$\texttt{amoeba}(H_j)
=\left\{\left(-\frac{1}{j}\log|m(1+u)^j-(m-1)u^j|,\log|u|\right)
:u\in\mathbb C\right\}.$$
Let $B$ be the component of $\mathbb R^2\backslash\texttt{amoeba}(H)$
containing a ray $(-\infty,b]\cdot(1,1)$. By Lemma \ref{amoebas}
we know that $B=B_1$, where $B_1$ is the component of
$\mathbb R^2\backslash\texttt{amoeba}(H_1)$ containing a ray
$(-\infty,b]\cdot(1,1)$. The strata are
$$S_1=\mathcal V_1\backslash\bigcup_{j=s+1}^r\mathcal V_{k_j},$$
$$S_{k_j}=\mathcal V_{k_j}\backslash\bigcup_{k_i\ne k_j}\mathcal V_{k_i},$$
and intersections of $\mathcal V_{k_i}$ and $\mathcal V_{k_j}$.
By Lemma \ref{amoebas}, only the critical points of $S_1$ may have log-moduli on
$$\partial B=\partial B_1=\{(-\log(m+u),\log u):u\in(0,\infty)\}.$$
The critical point on $S_1$ is described
by the critical point equations:
\begin{align*}
1-(m+u)z &= 0 \\
-hz(m+u) &= -nuz.
\end{align*}
The solution to the above system of equations is
$(z_*,u_*)=\left(\dfrac{n-h}{mn},\dfrac{hm}{n-h}\right)$, and its log-modulus
lies on $\partial B$. Since
$\mathbf x_{\min}=\mathrm{Re}\log\left(\dfrac{n-h}{mn},\dfrac{hm}{n-h}\right)$
is the unique minimizer in $\partial B$ for
$h=h_{\hat{\mathbf r}}$, i.e. $\mathbf x_{\min}$ minimizes
$-\hat{\mathbf r}\cdot\mathbf x$, and the singleton set containing
the critical point
$E=\left\{\left(\dfrac{n-h}{mn},\dfrac{hm}{n-h}\right)\right\}
\subseteq\mathbf T(\mathbf x_{\min})$ is a finite nonempty set of
quadratically nondegenerate smooth points, the intersection cycle
$$\sigma=\left[\sum_{\mathbf z\in W}\mathcal C(\mathbf z)\right]$$
is the sum of quasi-local cycles $\mathcal C(\mathbf z)$ for $\mathbf z\in E$,
where $\mathcal C(\mathbf z)$ is a homology generator of
$$(\mathcal V^{h(\mathbf x_{\min})+\varepsilon},
\mathcal V^{h(\mathbf x_{\min})-\varepsilon}),$$ for example the descending
submanifold.

We get asymptotics, so $[z^nu^h]\widehat{O}(z,u)$
\begin{align*}
\sim&(2\pi)^{\frac{1-2}{2}}\binom{-h}{s+1}(\det\mathcal H_1)^{-1/2}\cdot \frac{\left[G\middle/\displaystyle\prod_{j=s+1}^rH_{k_j}\right]
_{(z,u)=\left(\frac{n-h}{mn},\frac{hm}{n-h}\right)}}
{\left(\dfrac{hm}{n-h}\right)^{s+2}
\left(-\dfrac{n-h}{mn}\right)^{s+2}}\,h^{\frac{1-2}{2}}
\left(\frac{n-h}{mn}\right)^{-n}\left(\dfrac{hm}{n-h}\right)^{-h}\\
\sim&\frac{1}{\sqrt{2\pi}}\dfrac{(-h)^{s+1}}{(s+1)!}\frac{-h}{\sqrt{n(n-h)}}\cdot \frac{\left[G\middle/\displaystyle\prod_{j=s+1}^rH_{k_j}\right]
_{(z,u)=\left(\frac{n-h}{mn},\frac{hm}{n-h}\right)}}
{\left(-\frac{h}{n}\right)^{s+2}}\,{\frac{m^{n-h}}{\sqrt h}\left(1-\frac{h}{n}\right)^{h-n}
\left(\frac{n}{h}\right)^h},
\end{align*}
where
\begin{align*}
\det\mathcal H_1
&=\frac{Q}{(-uH_{1u})^3} \\
&=\frac{-u^2z^2z(-m-u)-u(-z)z^2(-m-u)^2-z^2u^2(-2)(-m-u)(-z)(-1)}{[(-u)(-z)]^3}\\
&=\frac{n(n-h)}{h^2} \text{ at } u = \frac{hm}{n-h},
\end{align*}
and where $\left[G\middle/\displaystyle\prod_{j=s+1}^rH_{k_j}\right]
_{(z,u)=\left(\frac{n-h}{mn},\frac{hm}{n-h}\right)}$
\begin{align*}
&=\left[\left(\frac{mn}{n-h}\right)^s
\left(\frac{n-h}{mn}\right)^k\right]\cdot \prod_{j=s+1}^r\frac{m\left(\dfrac{n-h+hm}{n-h}\right)^{k_j}
-(m-1)\left(\dfrac{hm}{n-h}\right)^{k_j}}
{1-\left[m\left(\dfrac{n-h+hm}{n-h}\right)^{k_j}
-(m-1)\left(\dfrac{hm}{n-h}\right)^{k_j}\right]
\left(\dfrac{n-h}{mn}\right)^{k_j}} \\
&=\prod_{j=s+1}^r\frac{m(n-h+hm)^{k_j}-(m-1)(hm)^{k_j}}
{(mn)^{k_j}-m(n-h+hm)^{k_j}+(m-1)(hm)^{k_j}}.
\end{align*}

Substituting the latter quantity, we get
\begin{align*}
&[z^nu^h]\widehat{O}(z,u)
\sim\frac{1}{\sqrt{2\pi}}\frac{(-h)^{s+1}}{(s+1)!}\frac{-h}{\sqrt{n(n-h)}}
\left(-\frac{n}{h}\right)^{s+2}
\frac{m^{n-h}}{\sqrt h}\left(1-\frac{h}{n}\right)^{h-n}\frac{n^h}{h^h}\cdot \\
&\prod_{j=s+1}^r\frac{m(n-h+hm)^{k_j}-(m-1)(hm)^{k_j}}
{(mn)^{k_j}-m(n-h+hm)^{k_j}+(m-1)(hm)^{k_j}} \\
&=\frac{m^{n-h}n^{s+1}(1-d)^{n(d-1)-\frac12}}
{(s+1)!\sqrt{2\pi nd}\,d^{nd}}
\prod_{j=s+1}^r\frac{[1+d(m-1)]^{k_j}-(1-1/m)(md)^{k_j}}
{m^{k_j-1}-[1+d(m-1)]^{k_j}+(1-1/m)(md)^{k_j}},
\end{align*}
where we let $d=h/n$ denote the density of holes.

Since the total number of partial words of length $n$ with $h$ holes over an alphabet of $m$ letters is,
by Stirling's approximation,
\begin{align*}
\binom{n}{h}m^{n-h}
&\sim m^{n-h}\left(\frac{n}{h}\right)^h\left(\frac{n}{n-h}\right)^{n-h}
\sqrt{\frac{2\pi n}{(2\pi h)2\pi(n-h)}} \\
&=\frac{m^{n-h}}{\sqrt{2\pi nd(1-d)}\,[d^d(1-d)^{1-d}]^n},
\end{align*}
we find that the mean number of occurrences of $p$ in a partial word
of length $n$ with hole density $d$ over an alphabet of $m$ letters is
$$\widehat{\Omega}_{n,d}\sim\frac{n^{s+1}}{(s+1)!}
\prod_{j=s+1}^r\frac{[1+d(m-1)]^{k_j}-\left(1-\frac1m\right)(md)^{k_j}}
{m^{k_j-1}-[1+d(m-1)]^{k_j}+\left(1-\frac1m\right)(md)^{k_j}}. %\qed
$$
\end{proof}

To illustrate Theorem~\ref{partialdensity}, consider the pattern $p=abacaba$, where $r=3$, $s=1$, $k_1=1$, $k_2=2$,
and $k_3=4$. Substituting these variables, $m=12$, $n=100$, and $d=1/10$ we find
that $$\widehat{\Omega}_{100,1/10}\approx 17.788\cdots.$$

When $\widehat{\Omega}_{n,d}<1$, we may apply Theorem \ref{expectation}, so in
that case we obtain the following corollary.

\begin{cor}
Suppose that a pattern $p$ uses $r$ distinct variables, where the $j$th variable
occurs $k_j\ge 1$ times. Without loss of generality, let $k=|p|=k_1+\cdots+k_r$
and $1=k_1=\cdots=k_s<k_{s+1}\le\cdots\le k_r$.
If $$n<(1+o(1))\left[(s+1)!
\prod_{j=s+1}^r\left(\frac{m^{k_j-1}}
{[1+d(m-1)]^{k_j}-\left(1-\frac1m\right)(md)^{k_j}}-1\right)
\right]^{\frac{1}{s+1}}$$
there is a partial word of length $n$ with hole density $d$ over an
alphabet of $m$ letters that avoids $p$.
\end{cor}

The upper bound for $L(m,Z_i)$ in Theorem~\ref{upperboundzimin} also applies to $L_d(m,Z_i)$. For a lower bound, we get the following.

\begin{cor}
$$L_d(m,Z_i)
\ge(1+o(1))\sqrt{2\prod_{j=1}^{i-1}\left(\frac{m^{2^j-1}}
{[1+d(m-1)]^{2^j}-\left(1-\frac1m\right)(md)^{2^j}}-1\right)}.$$
\end{cor}

Substituting for example $m=12$, $d=1/10$, and $i=3$, we find that
$$L_{1/10}(12,Z_3)\ge 23.709\cdots.$$

\section{Conclusion and open problems}

Using techniques from analytic combinatorics, we have calculated asymptotic
pattern occurrence statistics and used them in conjunction with the
probabilistic method to establish new results about Ramsey theoretic
pattern avoidance in the full word case (both nonabelian sense and abelian sense) and the partial word case. We have established, in particular, lower bounds for Ramsey lengths.

However, there may be more possible uses of these data in applications
such as cryptography and musicology. Cryptanalysts may compare the
pattern occurrence statistics of possible ciphertexts to those of random
noise to detect the existence of hidden messages; see the definitions
of semantic security and pseudorandom generator in \cite[pp.~67,~70]{katz}.
Musicologists may compare the pattern occurrence statistics of different
musical compositions to further their understanding of musical forms;
for previous work connecting music theory and theoretical computer science
see \cite{songs}.

We propose the following open problems.

\begin{prob}
Can you adapt the techniques appearing in this paper to the following two cases, which we have not considered, and get similar results?
\begin{itemize}
\item
When can a partial word avoid a given pattern in the abelian sense?
\item
When can a full necklace avoid a given pattern?
\end{itemize}
\end{prob}

\begin{prob}
Can you find better lower and upper bounds for the Ramsey lengths $L(m,p)$, $L_d(m,p)$, and $L_{\mathrm{ab}}(m,p)$ than the ones appearing in this paper?
\end{prob}

\begin{comment}
Referring to Theorem~\ref{abelian}, we propose the following open problem regarding analytic properties of
the power series $$S_{m,k}(z)=\displaystyle\sum_{\ell=1}^{\infty}
z^{k\ell}\sum_{i_1+\cdots+i_m=\ell}\binom{\ell}{i_1,\ldots,i_m}^k.$$

\begin{prob}
Let $m\ge 2$ and $k\ge 2$ be integers.
\begin{itemize}
\item What is the {\em exact} radius of convergence $R_k$ of
$S_{m,k}(z)$ when $k\ge 3$? In this paper, we have shown that
$R_1=R_2=1/m$ and $1/m\le R_k\le 1/m^{1/k}$ for $k\ge 3$.

\item Is $S_{m,k}(z)$ analytically continuable beyond its radius of convergence?

\item If the answer to the above question is yes, what is the
domain of holomorphy of $S_{m,k}(z)$?
\end{itemize}
\end{prob}

We also formulate the following conjecture.

\begin{conj}
The power series $S_{m,k}(z)$ can be analytically continued to yield
a meromorphic function, just like
\begin{center}
$\Gamma(z)=\int_0^{\infty}\!t^{z-1}e^{-t}\,\mathrm dt$
and $\zeta(s)=\sum_{\ell=1}^{\infty}\frac{1}{\ell^s}$.
\end{center}
\end{conj}
\end{comment}

Referring back to our quote by Hermann Weyl, we propose the following open
problem regarding pattern subgroups of symmetry groups on factors of words:
\begin{prob}
For an arbitrary integer $n>0$ and pattern $p$, what is the subgroup of the
symmetry group of $1^n$ that is generated by the $p$ pattern symmetries
of the factor $1^n$?
\end{prob}

%A World Wide Web server interface has been established at 
%\begin{center}
%{\tt www.uncg.edu/cmp/research/avoidability}
%\end{center}
%for automated use of a program that calculates the mean number of occurrences of a pattern in a word of a given length over a given alphabet size and gives a lower bound on the pattern's Ramsey length. 

\section*{Acknowledgements}

We thank Andrew Lohr from Rutgers University, Brent Woodhouse from The University of California at Los Angeles, and Francine Blanchet-Sadri from The University of North Carolina at Greensboro
for their very valuable comments and suggestions. We thank Francine Blanchet-Sadri for giving us support
to work under her supervision.

\bibliographystyle{plain}
\bibliography{bibliography}
\end{document}